\documentclass[a4paper,twoside]{article}      
\usepackage{amsmath,amssymb,amsfonts,amsthm,amscd,mathtools}
\usepackage{graphics}                 
\usepackage{color}                    
\usepackage{hyperref}                
\usepackage{ mathrsfs }
\usepackage{indentfirst}
\usepackage{bbold}
\usepackage{bm}
\usepackage{enumerate}
\usepackage{mathbbol}
\usepackage{url}         
\usepackage{colonequals} 
\usepackage{authblk} 
\usepackage{a4wide}
\usepackage{graphicx}
\usepackage{blkarray}

\hypersetup{
    bookmarks=true,         
    unicode=false,          
    pdftoolbar=true,        
    pdfmenubar=true,        
    pdffitwindow=false,     
    pdfstartview={FitH},    
    pdftitle={My title},    
    pdfauthor={Author},     
    pdfsubject={Subject},   
    pdfcreator={Creator},   
    pdfproducer={Producer}, 
    pdfkeywords={keywords}, 
    pdfnewwindow=true,      
    colorlinks=true,       
    linkcolor=blue,          
    citecolor=blue,        
    filecolor=magenta,      
    urlcolor=cyan           
}

\newtheorem{theorem}{Theorem}[section]
\newtheorem{proposition}[theorem]{Proposition}

\newtheorem{lemma}[theorem]{Lemma}
\theoremstyle{definition}
\newtheorem{remark}[theorem]{Remark}

\def\!{\mathop{\mathrm{!}}}

\DeclareMathOperator*{\E}{\mathbb{E}}

\usepackage{color}
\definecolor{lilas}{RGB}{182, 102, 210}

\numberwithin{equation}{section}

\def\be#1{\begin{equation*}#1\end{equation*}}
\def\ben#1{\begin{equation}#1\end{equation}}
\def\bea#1{\begin{eqnarray*}#1\end{eqnarray*}}
\def\bean#1{\begin{eqnarray}#1\end{eqnarray}}

\def\R{\mathbb{R}}
\def\N{\mathcal{N}}
\def\1{\mathbb{1}}
\def\P{\mathbb{P}}
\def\kE{\mathcal{E}}
\def\B{\mathbb{B}}
\def\C{\mathbb{C}}

\newcommand{\kO}{O}
\newlength{\boxwidth}
\DeclareUnicodeCharacter{2212}{-}
\title{On the expected number of real roots of random polynomials arising from evolutionary game theory}
\author[1,3]{V. H. Can}
\author[2] {M. H. Duong}
\author[1]{V. H. Pham}
\affil[1]{Institute of Mathematics, Vietnam Academy of Science and Technology, Vietnam.}
\affil[2]{School of Mathematics, University of Birmingham, UK.}
\affil[3]{Department of Statistics and Applied Probability,   National University of Singapore.}

\begin{document}
\maketitle
\begin{abstract}
In this paper, we obtain finite estimates and asymptotic formulas for the expected number of real roots of two classes of random polynomials arising from evolutionary game theory. As a consequence of our analysis, we achieve an asymptotic formula for the expected number of internal equilibria in multi-player two-strategy random evolutionary games. Our results contribute both to evolutionary game theory and random polynomial theory.
\end{abstract}
\section{Introduction}
\subsection{Motivation from evolutionary game theory and random polynomial theory}
Large random systems, in particular random polynomials and systems of random polynomials, arise naturally in a variety of applications in physics (such as in quantum chaotic dynamics \cite{Bogomolny1992}), biology (such as in theoretical ecology \cite{May1972}, evolutionary game theory and population dynamics \cite{GT10}), computer science (such as in the theory of computational complexity \cite{Shub1993}) and in social sciences (such as in social/complex networks \cite{Newman2003}). They are indispensable in the modelling and analysis of complex systems in which very limited information is available or where the environment changes so rapidly and frequently that one cannot describe the payoffs of their inhabitants’ interactions \cite{may2001stability,fudenberg1992evolutionary,GT10,gross2009generalized,Galla2013}. The study of statistics of equilibria in large random systems provides important insight into the understanding of the underlying physical, biological and social system such as the complexity-stability relationship in ecosystems \cite{May1972,gross2009generalized,Pimm1984,FyoKho2016}, bio-diversity and maintenance of polymorphism in multi-player multi-strategy games  \cite{GT10}, and the learning dynamics \cite{Galla2013}.  A key challenge in such study is due to the large (but finite) size of the underlying system (such as the population in an ecological system, the number of players and strategies in an evolutionary game and the number of nodes and connections in a social network). Understanding the behaviour of the system at finite size or characterizing its asymptotic behaviour when the size tends to infinity are of both theoretical and practical interest, see for instance \cite{Pereda2019, PENA2018}.

In this paper we are interested in the number of internal equilibria in $(n+1)$-player two-strategy random evolutionary games as in \cite{DH15,DuongHanJMB2016,DuongTranHanDGA, DuongTranHanJMB}. We consider an infinitely large population that consists of individuals using two strategies, A and B. We denote by $y$, $0 \leq y\leq 1$, the frequency of strategy A in the population. The frequency of strategy B is thus $(1-y)$. The interaction of the individuals in the population is in randomly selected groups of $(n+1)$ participants, that is, they interact and obtain their fitness from $(n+1)$-player games. In this paper, we consider symmetric games where the payoffs do not depend on the ordering of the players. Suppose that $a_i$ (respectively, $b_i$) is the payoff that an A-strategist (respectively, B) achieves when interacting with a group of $n$ other players consisting $i$ ($0\leq i\leq n$) A strategists and $(n-i)$ B strategists. In other words, the payoff matrix is given by
\begin{equation*}
\begin{blockarray}{ccccccc}\hline
\text{Opossing A players} &0 & 1&\ldots & i & \ldots & n \\ \hline
\begin{block}{ccccccc}
  \text{A} & a_0 & a_1 & \ldots & a_i&\ldots & a_n  \\
  \text{B} & b_0 & b_1 & \ldots & b_i &\ldots & b_n\\
 \end{block}
 \hline
\end{blockarray}
\end{equation*}
The average payoffs (fitnesses) of strategies A and B are respectively given by
\begin{equation*}
\pi_A= \sum\limits_{i=0}^{n}a_i\begin{pmatrix}
n\\
i
\end{pmatrix}y^i (1-y)^{n-i}	\quad\text{and}\quad
\pi_B = \sum\limits_{i=0}^{n}b_i\begin{pmatrix}
n\\
i
\end{pmatrix}y^i(1-y)^{n-i}.
\end{equation*}
Internal equilibria in $(n+1)$-player two-strategy games can be derived using the replicator dynamic approach~\cite{GT10} or the definition of an evolutionary stable strategy, see e.g., \cite{broom:1997aa}. They are those points $0<y<1$ (note that $y=0$ and $y=1$ are trivial equilibria in the replicator dynamics) such that the fitnesses of the two strategies are the same $\pi_A=\pi_B$, that is
\begin{equation*}
\sum\limits_{i=0}^{n}\xi_i \begin{pmatrix}
n\\
i
\end{pmatrix}y^i (1-y)^{n-i}=0\quad\text{where}\quad \xi_i=a_i-b_i.
\end{equation*}
In the literature, the sequence of the difference of payoffs $\{\xi_i\}_{i}$ is called the gain sequence \cite{Bach2006, Pena2014}. Dividing the above equation by $(1-y)^{n}$ and using the transformation $x=\frac{y}{1-y}$, we obtain the following polynomial equation for $x$ ($x>0$)
\begin{equation}
\label{eq: P}
P_n(x):=\sum\limits_{i=0}^{n}\xi_i\begin{pmatrix}
n\\
i
\end{pmatrix}x^i=0,
\end{equation}
In random games, the payoff entries $\{a_i\}_i$ and $\{b_i\}$ are random variables, thus so are the gain sequence $\{\xi_i\}_i$. Therefore, the expected number of internal equilibria in a $(n+1)$-player two-strategy random game is the same as the expected number of positive roots of the random polynomial $P_n$, which is half of the expected number of the real roots of $P_n$ due to the symmetry of the distributions. This connection between evolutionary game theory and random polynomial theory has been revealed and exploited in recent serie of papers \cite{DH15,DuongHanJMB2016,DuongTranHanDGA, DuongTranHanJMB}. It has been shown that, if $\{\xi_i\}_i$ are i.i.d normal (Gaussian) distributions then  \cite{DH15,DuongHanJMB2016}
\begin{equation}
\label{eq: finite1}
\frac{2n}{\pi\sqrt{2n-3}}\leq \mathbb{E} N_n\leq \frac{2\sqrt{n}}{\pi}\Big(1+\ln 2+\frac{1}{2}\ln (n)\Big)\quad \forall n,
\end{equation}
where $N_n$ is the number of real roots of $P_n$. We emphasize that \eqref{eq: finite1} is true for all finite group size $n$, which is useful for practical purposes, for instance when doing simulations. A direct consequence of this estimate is the following asymptotic limit
\begin{equation}
\label{eq: limitn}
\lim\limits_{n\rightarrow\infty}\frac{\ln \E N_n}{\ln n}=\frac{1}{2}.
\end{equation}
On the other hand, the expected number of real roots of random polynomials has been a topic of intensive research over the last hundred years. Three most well-known classes studied in the literature are 
\begin{enumerate}[(i)]
\item Kac polynomials: $\sum_{i=0}^n \xi_i x^i$,
\item Weyl (or flat) polynomials: $\sum_{i=0}^n \frac{1}{i!}\xi_i  x^i$,
\item Elliptic (or binomial) polynomials: $\sum_{i=0}^n\sqrt{\begin{pmatrix}
n\\
i
\end{pmatrix}} \xi_i x^i$.
\end{enumerate}
When $\{\xi_i\}$ are Gaussian distributions, it has been proved, see for instance \cite{EK95}, that
\begin{equation}
\label{eq: others}
\E N_n=\begin{cases}
\frac{2}{\pi}\ln (n)+C_1+\frac{2}{n\pi}+O(1/n^2)\quad&\text{for Kac polynomials},\\
\sqrt{n}\quad&\text{for elliptic polynomials},\\
\sqrt{n}\Big(\frac{2}{\pi}+o(1)\Big)\quad&\text{for Weyl polynomials}.
\end{cases}
\end{equation}
We refer the reader to standard monographs \cite{BS86,farahmand1998} for a detailed account and \cite{NNV2016,Do2018} for recent developments of the topic. The asymptotic formulas \eqref{eq: others} are much stronger than the limit \eqref{eq: limitn} because they provide precisely the leading order of the quantity $\mathbb{E}N_n$.  A natural question arises: \textit{can one obtain an asymptotic formula akin to \eqref{eq: others} for the random polynomial from random multi-player evolutionary games?} It has been conjectured, in a study on computational complexity, by Emiris and Galligo \cite{Emiris:2010} and formally shown in \cite{DuongTranHanDGA} that
\begin{equation}
\label{eq: limitn2}
\E N_n\sim \sqrt{2n}+O(1).
\end{equation} 
In this paper, we rigorously prove generalizations of the asymptotic formula \eqref{eq: limitn2} and of the finite group size estimates \eqref{eq: finite1} for two more general classes of random polynomials
\begin{equation}
\label{eq: generalP}
P_n^{(\gamma)}(x)=\sum_{i=0}^n \xi_i \begin{pmatrix}
n\\i
\end{pmatrix}^\gamma  x^i\quad\text{and}\quad P^{(\alpha,\beta)}_n(x)=\sum_{i=0}^n \begin{pmatrix}
n+\alpha\\  n-i
\end{pmatrix}^\frac{1}{2}\begin{pmatrix}
n+\beta\\  i
\end{pmatrix}^\frac{1}{2} \xi_i  x^i. 
\end{equation}
Here $\gamma>0, \alpha, \beta>-1$ are given real numbers, $\{\xi_i\}_{i=0,\ldots,n}$ are standard normal i.i.d. random variables. The class of random polynomials $P_n$ arising from evolutionary game theory is a special case of both $P_n^{(\gamma)}$ (when $\gamma=1$) and $P_n^{(\alpha, \beta)}$ (when $\alpha=\beta=0$). For general values of $\alpha, \beta$ and $\gamma$, $P_n^{(\gamma)}$ and $P_n^{(\alpha, \beta)}$ are related to more complex models in evolutionary game theory where the gain sequence $\{\xi_i\}_i$ depends not only on $i$ but also on group size $n$. An example for such scenarios is in a public goods game in which the benefit from cooperation are shared among all group members rather than accruing to each individual \cite{Hauert2006, Pacheco2009, PENA2018}. From a mathematical point of view, the class $P^{(\gamma)}_n$ is a natural extension of $P_n$ and covers both Kac polynomials and elliptic polynomials as special cases (corresponding to $\gamma=0$ and $\gamma=\frac{1}{2}$ respectively). In addition, as previously shown in \cite{DuongHanJMB2016}, $P_n$ is connected to Legendre polynomials. As will be shown in Section \ref{sec: finite estimate}, the class $P_n^{(\alpha,\beta)}$ is intrinsically related to Jacobi polynomials, which contain Legendre polynomials as special cases. The link between $P_n$ and Legendre polynomials in \cite{DuongHanJMB2016} is extended to that of between $P_n^{(\alpha,\beta)}$ and Jacobi polynomials in the present paper.

\subsection{Main results}
Throughout this paper, we suppose that $\{\xi_i\}$ are i.i.d standard normal distributions. We denote by $\E N_n^{(\gamma)}$ and $\E N_n^{(\alpha,\beta)}$ the expected number of real roots of $P_n^{(\gamma)}$ and $P_n^{(\alpha,\beta)}$ respectively. The main results of the present paper are the following theorems.

\begin{theorem}[Estimates of $E N_n^{(\alpha,\beta)}$ for any $n$]
\label{thm: finite n estimates}
\begin{enumerate}[(1)] Suppose that $\alpha, \beta>-1$.
\item (estimates in terms of roots of Jacobi polynomial) Let $0<s_{n,max}<1$ be the maximum root of the Jacobi's polynomial of degree $n$ as defined in (\ref{eq: Jacobi}) . Then 
\begin{equation}
\sqrt{n}\frac{1-s_{n,max}}{1+s_{n,max}}\leq \E N_n^{(\alpha,\beta)}\leq \sqrt{n}\frac{1+s_{n,max}}{1-s_{n,max}}.
\end{equation}
\item (explicit estimates for finite $n$) For all $\alpha=\beta>-1$, it holds that
\begin{equation}
\frac{2}{\pi}\sqrt{\frac{n(n+2\alpha)}{2n+2\alpha-1}}\leq \E N_n^{(\alpha,\alpha)}\leq \frac{2\sqrt{n}}{\pi}\Big(1+\ln(2)+\frac{1}{2}\log\frac{n+\alpha}{1+\alpha}\Big).
\end{equation}
\end{enumerate}
\end{theorem}
Theorem \ref{thm: finite n estimates}, which combines Theorems \ref{thm: Jacobi-estimate} and \ref{thm: ultraspherical}, provides lower and upper bounds for $E N_n^{(\alpha,\beta)}$ in terms of the group size $n$. It is only applicable to the class $P_n^{\alpha,\beta}$ since our proof makes use of a connection between $P_n^{\alpha,\beta}$ and Jacobi polynomials. In addition, in the second part, we use a symmetry condition on the coefficients of the polynomial $P_n^{(\alpha,\beta)}$ which requires $\alpha=\beta$. The next result characterizes the asymptotic limits, as the group size $n$ tends to infinity, of both $\E N_n^{(\gamma)}$ and $\E N_n^{(\alpha, \beta)}$.
\begin{theorem}[Asymptotic behaviour as $n\rightarrow +\infty$]
\label{thm: asymptotic}
We have
\begin{equation}
\label{eq: main asymptotic behaviour}
\E N_n^{(\gamma)}\sim \sqrt{2\gamma n}(1+o(1))\quad\text{and}\quad \E N_n^{(\alpha,\beta)}\sim \sqrt{2 n}(1+o(1))\quad \text{as}~~n\rightarrow \infty.
\end{equation}
As a consequence, there is a phase transition (discontinuity) in the expected number of roots of $\mathbb{E} N_n^{(\gamma)}$ as a function of $\gamma$ as $n\rightarrow \infty$
\begin{equation}
\label{eq: phase transition}
\E N_n^{(\gamma)}\sim \begin{cases}
\frac{2}{\pi} \ln (n)\quad \text{for}~~\gamma=0,\\
\sqrt{2\gamma n}\quad \text{for}~~ \gamma>0.
\end{cases}
\end{equation}
\end{theorem}
Our study on the expected number of real roots of $P_n^{(\gamma)}$ and $P_n^{(\alpha,\beta)}$ contributes to both evolutionary game theory and random polynomial theory. From an evolutionary game theory point of view, our results show surprisingly that in random multiplayer evolutionary games, one expects much less number of equilibria, which is proportional to the square root of the group size, than in deterministic games (recalling that the expected number of internal equilibria is the same as the expected number of positive roots, which is half of the expected number of real roots). In addition, since for a polynomial equation, the number of stable equilbria is half of that of equilibria, our results also apply to stable equilibria. From a random polynomial theory point of view, the present paper introduces two meaningful classes of random polynomials that have not been studied in the literature. In particular, the fact that the asymptotic behavour of $\E N_n^{(\alpha,\beta)}$ is independent from $\alpha$ and $\beta$ is rather unexpected and is interesting on its own right. In addition the phase transition phenomenon \eqref{eq: phase transition}, to the best of our knowledge, is shown for the first time.

\subsection{Organization of the paper}
The rest of the paper is organized as follows. In Section \ref{sec: Kac-Rice} we recall the Kac-Rice formula for computing the expected number of real roots of a random polynomial. In Section \ref{sec:finite and asymptotic of ENalpha}, we establish connections between $P_n^{(\alpha,\beta)}$ and Jacobi polynomials and prove Theorem \ref{thm: finite n estimates}. Proof of Theorem \ref{thm: asymptotic} is presented in Section \ref{sec: asymptotic results} and Section \ref{sec: asymptotic of ENalpha}. In Section \ref{sec: summary} we provide further discussions and outlook. Finally, detailed proofs of technical lemmas are given in Appendix \ref{sec: Appendix}.

\section{Kac-Rice formula}
\label{sec: Kac-Rice}
In this section, we recall the celebrated Kac-Rice formula for computing the expected number of real roots of a random polynomials, which is the starting point of our analysis. Consider a general random polynomial
$$
p_n(x)=\sum_{i=0}^n a_i \xi_i x^i.
$$
Let $\{\xi\}$ are standard i.i.d. random variables. Let $\E N_n(a,b)$ be the expected number of real roots of $p_n$ in the interval $(a,b)$. Then the Kac-Rice formula is given by, see for instance \cite{EK95}
\begin{equation}
\label{eq: Kac-Rice general}
\mathbb{E}N_n(a,b)=\frac{1}{\pi}\int_a^b \frac{\sqrt{A_n(x)M_n(x)-B_n^2(x)}}{M_n(x)}\,dx
\end{equation}
where
$$
M_n(x)=\mathrm{var}(p_n(x)),\quad A_n(x)=\mathrm{var}(p_n'(x)),\quad B=\mathrm{cov}(p_n(x)p_n'(x)).
$$
We can find $M_n, A_n$ and $B_n$ explicitly in terms of the coefficients $\{a_i\}$ of $p_n$ as follows. Since $\{\xi_i\}$ are standard i.i.d. random variables, we have
\begin{align*}
&p_n'(x)=\sum_{i=0}^n a_i i \xi_i x^{i-1},\quad
p_n(x)^2=\sum_{i,j=0}^n a_ia_j \xi_i\xi_j x^{i+j},\quad p_n(x)p_n'(x)=\sum_{i,j=0}^n a_ia_j i\xi_i\xi_j x^{i+j-1},\\
&\mathbb{E}(p_n(x))=\sum_{i=0}^n a_ix^i\mathbb{E}(\xi_i)=0, \quad \mathbb{E}(p_n'(x))=0
\\&M_n(x)=\mathrm{var}(p_n(x))=\mathbb{E}(p_n^2(x))-(\mathbb{E}(p_n(x)))^2=\sum_{i,j=0}^n a_ia_jx^{i+j}\mathbb{E}(\xi_i\xi_j)=\sum_{i=0}^n a_i^2 x^{2i},
\\&A_n(x)=\mathrm{var}(p_n'(x))=\mathbb{E}((p_n'(x))^2)-(\mathbb{E}(p_n'(x)))^2=\sum_{i,j=0}^n a_i a_j ij x^{i+j-2}\mathbb{E}(\xi_i\xi_j)=\sum_{i=0}^n a_i^2 i^2 x^{2(i-1)},
\\& B_n(x)=\mathrm{cov}(p_n(x)p_n'(x))=\mathbb{E}(p_n(x)p_n'(x))=\sum_{i,j=0}^n i a_i a_j x^{i+j}\mathbb{E}(\xi_i\xi_j)=\sum_{i=0}^n i\, a_i^2 x^{2i-1}.
\end{align*}
In conclusion, we have
\begin{equation}
\label{eq: A, B, M}
M_n(x)=\sum_{i=0}^n a_i^2 x^{2i},\quad A_n(x)= \sum_{i=0}^n a_i^2 i^2 x^{2(i-1)},\quad B_n(x)=\sum_{i=0}^n i\, a_i^2 x^{2i-1}.
\end{equation}
Furthermore, the following relations between $M_n, A_n$ and $B_n$, which follow directly from the above formulas, will also be used in the subsequent sections:
\begin{align*}
B_n(x)&=\frac{1}{2}M_n'(x),\quad A_n(x)=\frac{1}{4x}\Big(xM_n'(x)\Big)',
\\ \frac{A_n(x)M_n(x)-B_n^2(x)}{M_n^2(x)}&=\frac{1}{4}\Big(\frac{M_n''(x)}{M_n(x)}+\frac{1}{x}\frac{M_n'(x)}{M_n(x)}-\Big(\frac{M_n'(x)}{M_n(x)}\Big)^2\Big)
\\&=\frac{1}{4}\Big(\frac{1}{x}\frac{M_n'(x)}{M_n(x)}+\Big(\frac{M_n'(x)}{M_n(x)}\Big)'\Big)=\frac{1}{4x}\Big(x\frac{M_n'(x)}{M_n(x)}\Big)',
\end{align*}
where the prime $'$ notation denotes a derivative with respect to the variable $x$. 

Let $\E N_n^{(\gamma)}(a,b)$ and $\E N_n^{(\alpha,\beta)}(a,b)$ be  respectively the expected number of real roots of $P_n^{(\gamma)}$ and of $P_n^{(\alpha,\beta)}$ in a given interval $[a,b]$. Applying \eqref{eq: Kac-Rice general}-\eqref{eq: A, B, M} to $P_n^{(\gamma)}$ and to $P_n^{(\alpha,\beta)}$, we obtain the following common formula for $\E N_n^{(\gamma)}(a,b)$ and $\E N_n^{(\gamma)}(a,b)$ but with different triples $\{A_n, B_n, M_n\}$
\begin{equation}
\label{eq: formula EN}
\E N_n^{(*)}(a,b)=\frac{1}{\pi}\int_{a}^b \frac{\sqrt{A_n(x) M_n(x)-B_n^2(x)}}{M_n(x)}\,dx,
\end{equation}
where $(*)\in\{(\gamma), (\alpha,\beta)\}$.
For $EN_n^{(\gamma)}(a,b)$:
\begin{equation}
\label{eq: MAB}
M_n(x)=\sum_{k=0}^{n}\begin{pmatrix}
n\\k
\end{pmatrix}^{2\gamma} x^{2k}, 
\; A_n(x)=\sum_{k=0}^{n}k^2\begin{pmatrix}
n\\k
\end{pmatrix}^{2\gamma}x^{2(k-1)}, 
\; B_n(x)=\sum_{k=0}^n k \begin{pmatrix}
n\\k
\end{pmatrix}^{2\gamma} x^{2k-1}.
\end{equation}
For $E N_n^{(\alpha,\beta)}$:
\begin{align}
\label{eq: MAB2}
&M_n(x)=\sum_{k=0}^{n}\begin{pmatrix}
n+\alpha\\ n-k
\end{pmatrix}\begin{pmatrix}
n+\beta\\ k
\end{pmatrix} x^{2k}, 
\; A_n(x)=\sum_{k=0}^{n}k^2\begin{pmatrix}
n+\alpha\\ n-k
\end{pmatrix}\begin{pmatrix}
n+\beta\\ k
\end{pmatrix}x^{2(k-1)}, \notag\\ 
& B_n(x)=\sum_{k=0}^n k \begin{pmatrix}
n+\alpha\\ n-k
\end{pmatrix}\begin{pmatrix}
n+\beta\\ k
\end{pmatrix} x^{2k-1}.
\end{align}
By writing $\E N_n^{(\gamma)}$ or $\E N_n^{(\alpha,\beta)}$ it becomes clear which class of random polynomials is under consideration; therefore, for notational simplicity, we simply write $\{A_n, B_n, M_n\}$ without superscripts $(\gamma)$ or $(\alpha,\beta)$. The above Kac-Rice formulas are starting points for our analysis. The difficulty now is to analyze the integrand in \eqref{eq: formula EN} for each class of random polynomials.
\section{Finite group-size estimates}
\label{sec:finite and asymptotic of ENalpha}
In this section, we show a connection between the class $P_n^{(\alpha,\beta)}$ and Jacobi polynomials which extends that of between $P_n$ and Legendre polynomials in \cite{DuongHanJMB2016}. Using this connection, we will prove Theorem \ref{thm: finite n estimates} on the estimates of $\E N_n^{(\alpha,\beta)}$ for finite $n$.
\subsection{Connections to Jacobi polynomials and finite estimates of $\E N_n^{(\alpha,\beta)}$}
\label{sec: finite estimate}
We recall that the Jacobi polynomial is given by 
\begin{equation}
\label{eq: Jacobi}
J^{(\alpha,\beta)}_n(x)=\sum_{i=0}^n \begin{pmatrix}
n+\alpha\\n-i
\end{pmatrix}\begin{pmatrix}
n+\beta\\ i
\end{pmatrix}\Big(
\frac{x-1}{2}
\Big)^i\Big(
\frac{x+1}{2}
\Big)^{n-i}.
\end{equation}
If $\alpha=\beta$, Jacobi's polynomial $J_n^{(\alpha,\beta)}(x)$ is called an ultraspherical polynomial. Legendre's polynomial is a special case of Jacobi's polynomial when $\alpha=\beta=0$. It is well-known that the zeros of $J_n^{(\alpha,\beta)}$ are real, distinct and are located in the interior of the interval $[-1,1]$ \cite{szego1975book}. The following lemma links $M_n^{(\alpha,\beta)}$ to Jacobi polynomials. Its proof is given in Appendix \ref{sec: Appendix}.
\begin{lemma}
\label{lem: relation Mn vs Jn} It holds that
\begin{equation}
M^{(\alpha,\beta)}_n(x)=(1-x^2)^n J^{(\alpha,\beta)}_n\Big(\frac{1+x^2}{1-x^2}\Big).
\end{equation}
\end{lemma}
The following theorem provides estimates of $\mathbb{E}(\N_\R)$ in terms of the maximum root of the Jacobi polynomial.
\begin{theorem}
\label{thm: Jacobi-estimate}
Let $0<s_{n,max}<1$ be the maximum root of the Jacobi's polynomial of degree $n$. Then the expected number of real roots, $\mathbb{E} N_n^{(\alpha,\beta)}$, of $P^{(\alpha,\beta)}_n$ satisfies
\begin{equation}
\sqrt{n}\frac{1-s_{n,max}}{1+s_{n,max}}\leq \mathbb{E} N_n^{(\alpha,\beta)}\leq \sqrt{n}\frac{1+s_{n,max}}{1-s_{n,max}}.
\end{equation}
\end{theorem}
\begin{proof}
Let $\{-1<s_1<s_2<\ldots<s_ n<1\}$ be the zeros of the Jacobi polynomial of degree $n$. Note that $s_k=-s_{n+1-k}<0$ for $k=1,\ldots, \lfloor\frac{n}{2}\rfloor$. We deduce from Lemma \ref{lem: relation Mn vs Jn} that $M_n$ has $2n$ distinct zeros given by $\{\pm i\sqrt{\frac{1-s_k}{1+s_k}},~~1\leq k\leq n\}$ which are purely imaginary. Thus $M_n$ can be written as
\begin{equation}
\label{eq: representation of Mn}
M_n(x)=m_n \prod_{k=1}^n (x^2+r_k),
\end{equation}
where $m_n$ is the leading coefficient and for $1\leq k\leq n$
\begin{equation}
\label{eq: roots of M vs root of Jacobi}
r_k=\frac{1-s_k}{1+s_k}>0.
\end{equation}

It follows from the properties of $\{s_k\}$ that
$r_1>r_2>\ldots>r_n>0$ and $r_{k}r_{n+1-k}=1$ for $k=1,\ldots, \lfloor\frac{n}{2}\rfloor$. 
Using the representation \eqref{eq: representation of Mn} of $M_n$ we have
\begin{align*}
M_n'(x)=2x m_n\sum_{k=1}^n\prod_{j\neq k} (x^2+r_j), \quad \frac{M_n'(x)}{M_n(x)}=\sum_{k=1}^n\frac{2x}{x^2+r_k},\quad 
\Big(x\frac{M_n'(x)}{M_n(x)}\Big)'=\sum_{k=1}^n\frac{4 x r_k}{(x^2+r_k)^2}.
\end{align*}
Hence the density function can be represented as
\begin{align}
\label{eq: f in terms of roots of Mn}
f_n(x)^2=\frac{1}{4x}\Big(x\frac{M_n'(x)}{M_n(x)}\Big)'=\sum_{k=1}^n\frac{r_k}{(x^2+r_k)^2}.
\end{align}
Since $0<r_n<\ldots<r_1$, we deduce that
\begin{equation}
n \frac{r_n}{(x^2+r_1)^2}\leq f_n(x)^2=	\sum_{k=1}^n\frac{r_k}{(x^2+r_k)^2}\leq n\frac{r_1}{(x^2+r_n)^2},
\end{equation}
that is
$$
\sqrt{n}\frac{\sqrt{r_n}}{x^2+r_1}\leq f_n(x)\leq \sqrt{n}\frac{\sqrt{r_1}}{x^2+r_n}.
$$
Since
$$
\mathbb{E} N_n^{(\alpha,\beta)}=\frac{1}{\pi}\int_{-\infty}^\infty f_n(x)\,dx
$$ 
we have
\begin{equation*}
\frac{1}{\pi}\int_{-\infty}^{\infty}\frac{\sqrt{n r_n}}{x^2+r_1}\,dx\leq \mathbb{E} N_n^{(\alpha,\beta)}\leq \frac{1}{\pi}\int_{-\infty}^{\infty}\frac{\sqrt{n r_1}}{x^2+r_n}\,dx, 
\end{equation*}
that is, since $\int_{-\infty}^\infty\frac{1}{x^2+a}\,dx=\frac{\pi}{\sqrt{a}}$ for $a>0$,
\begin{equation*}
\sqrt{n}\sqrt{\frac{r_n}{r_1}}\leq \mathbb{E} N_n^{(\alpha,\beta)}\leq \sqrt{n}\sqrt{\frac{r_1}{r_n}}.
\end{equation*}
Since $r_1r_n=1$, the above expression can be written as
$$
\sqrt{n}r_n\leq \mathbb{E} N_n^{(\alpha,\beta)}\leq \sqrt{n}r_1.
$$
From \eqref{eq: roots of M vs root of Jacobi}, we obtain the following estimate for $\mathbb{E} N_n^{(\alpha,\beta)}$ in terms of roots of Jacobi's polynomials
\begin{equation*}
\sqrt{n}\frac{1-s_n}{1+s_n}=\sqrt{n}\frac{1+s_1}{1-s_1}\leq \mathbb{E} N_n^{(\alpha,\beta)}\leq \sqrt{n}\frac{1-s_1}{1+s_1}=\sqrt{n}\frac{1+s_n}{1-s_n}.
\end{equation*}
This completes the proof of the theorem. 
\end{proof}
The following theorem provides an explicit finite estimate for $\mathbb{E} N_n^{(\alpha,\beta)}$ in the ultraspherical case. It generalizes a previous result for $\alpha=0$ (see \eqref{eq: finite1}) obtained in \cite{DuongHanJMB2016}.
\begin{theorem}
\label{thm: ultraspherical}
Consider the ultraspherical case (i.e., $\alpha=\beta$). We have
\begin{equation}
\frac{2}{\pi}\sqrt{\frac{n(n+2\alpha)}{2n+2\alpha-1}}\leq \mathbb{E} N_n^{(\alpha,\alpha)}\leq \frac{2\sqrt{n}}{\pi}\Big(1+\ln(2)+\frac{1}{2}\log\frac{n+\alpha}{1+\alpha}\Big).
\end{equation}
As a consequence,
\begin{equation}
\lim\limits_{n\rightarrow+\infty}\frac{\ln(\mathbb{E} N_n^{(\alpha,\alpha)})}{\ln(n)}=\frac{1}{2}.
\end{equation}
\end{theorem}
\begin{proof}
Since $\alpha=\beta$ changing $x$ to $1/x$ and $x$ to $−x$ leaves the distribution of the coefficients of $P_n^{(\alpha,\alpha)}(x)$ invariant. Thus we obtain that
$$
\mathbb{E} N_n^{(\alpha,\beta)}=4\mathbb{E} N_n^{(\alpha,\beta)}(-\infty,-1)=4\mathbb{E} N_n^{(\alpha,\beta)}(-1,0)=4\mathbb{E} N_n^{(\alpha,\beta)}(0,1)=4\mathbb{E} N_n^{(\alpha,\beta)}(1,\infty).
$$
It follows from \eqref{eq: f in terms of roots of Mn} that $f_n(x)$ is decreasing on $(0,+\infty)$. Thus for any $x\in[0,1]$, we have
\begin{equation}
\label{eq: estimate1}
f_n(0)=\sqrt{\frac{n(n+\alpha)}{1+\alpha}}\geq f_n(x)\geq f_n(1)=\frac{1}{2}\sqrt{\frac{n(n+2\alpha)}{2n+2\alpha-1}}.
\end{equation}
In addition, since $(x^2+r_k)^2\geq 4r_k x^2$ for all $x>0$, we also deduce from \eqref{eq: f in terms of roots of Mn} that
$$
f_n(x)^2\leq \frac{n}{4 x^2} \quad\text{for}~~x>0,
$$
that is 
\begin{equation}
\label{eq: estimate 2}
f_n(x)\leq \frac{\sqrt{n}}{2x}\quad\text{for}~~x>0.
\end{equation}
Using the second inequality in \eqref{eq: estimate1} we obtain the lower bound for $\mathbb{E} N_n^{(\alpha,\beta)}$ as follows
$$
\mathbb{E} N_n^{(\alpha,\beta)}=\frac{4}{\pi}\int_{0}^{1}f_{n}(x)\,dx\geq \frac{4}{\pi}\int_0^1 f_n(1)\,dx=\frac{4f_n(1)}{\pi}=\frac{2}{\pi}\sqrt{\frac{n(n+2\alpha)}{2n+2\alpha-1}}.
$$
Using the first inequality in \eqref{eq: estimate1} and \eqref{eq: estimate 2} we obtain the following upper bound for $\mathbb{E} N_n^{(\alpha,\beta)}$ for any $0<\gamma<1$
\begin{align*}
\mathbb{E} N_n^{(\alpha,\beta)}&=\frac{4}{\pi}\int_{0}^{1}f_{n}(x)\,dx=\frac{4}{\pi}\Big(\int_{0}^{\gamma}f_{n}(x)\,dx+\int_{\gamma}^{1}f_{n}(x)\,dx\Big)
\\&\leq \frac{4}{\pi}\Big(\int_{0}^{\gamma}f_{n}(0)\,dx+\int_{\gamma}^{1}\frac{\sqrt{n}}{2x}\,dx\Big)
\\&=\frac{4}{\pi}\Big(\gamma\sqrt{\frac{n(n+\alpha)}{1+\alpha}}-\frac{\sqrt{n}}{2}\ln(\gamma)\Big).
\end{align*}
We choose $\gamma\in(0,1)$ that minimizes the right-hand side of the above expression. That is
$$
\gamma=\frac{1}{2}\sqrt{\frac{1+\alpha}{n+\alpha}},
$$
which gives
$$
\mathbb{E} N_n^{(\alpha,\beta)}\leq \frac{2\sqrt{n}}{\pi}\Big(1+\ln(2)+\frac{1}{2}\log\frac{n+\alpha}{1+\alpha}\Big).
$$
This completes the proof of the theorem.
\end{proof}

\section{Asymptotic behaviour of $E N_n^{(\gamma)}$}
\label{sec: asymptotic results}
In this section, we prove Theorem \ref{thm: asymptotic} obtaining asymptotic formulas for $\E N_n^{(\gamma)}$. 

\noindent\textbf{Strategy of the the proof}. Let us first explain the main idea of the proof since it requires a rather delicate analysis. The first observation is that, similarly as the proof of Theorem \ref{thm: ultraspherical}, since changing $x$ to $1/x$ and $x$ to $−x$ leaves the distribution of the coefficients of $P_n^{(\gamma)}(x)$ invariant, we have 
$$
\E N_n^{(\gamma)}(-\infty,-1)=\E N_n^{(\gamma)}(-1,0)=\E N_n^{(\gamma)}(0,1)=\E N_n^{(\gamma)}(1,\infty).
$$
Thus $\E N_n^{(\gamma)}=4\E N_n^{(\gamma)}(0,1)$ and it suffices to analyze $\E N_n^{(\gamma)}(0,1)$. We then split the interval $(0,1)$ further into two smaller intervals $(0,\eta)$ and $(\eta, 1)$,  $\E N_n^{(\gamma)}(0,1)=\E N_n^{(\gamma)}(0,\eta)+\E N_n^{(\gamma)}(\eta,1)$ for a carefully chosen $0<\eta<1$ (which may depend on $n$) such that $\E N_n^{(\gamma)}(0,\eta)$ is negligible. To select a suitable $\eta$, we will use Jensen's inequality (see Lemma \ref{lem: Jensen}) that provides an upper bound  on  the  number  of  roots  of  an analytic function (including polynomials) in an open ball. As such, we obtain $\eta=n^{-3\gamma/4}$ and write
$$
\E N_n^{(\gamma)}(0,1)=\E N_n^{(\gamma)}(0,n^{-3\gamma/4})+\E N_n^{(\gamma)}(n^{-3\gamma/4},1).
$$
In fact as will be shown, $\E N_n^{(\gamma)}(0,n^{-3\gamma/4})$ is of order $o(\sqrt{n})$, which is negligible (see Proposition \ref{prop: fisrt interval}). The next step is to obtain precisely the leading order in $\E N_n^{(\gamma)}(n^{-3\gamma/4},1)$. We recall that by Kac-Rice formula (see Section \ref{sec: Kac-Rice}) we have
\begin{equation}
\label{eq: temp1}
\E N_n^{(\gamma)}(n^{-3\gamma/4},1)=\int_{n^{-3\gamma/4}}^1\frac{\sqrt{A_n(x)M_n(x)-B^2_n(x)}}{M_n(x)}\,dx
\end{equation}
where $A_n, B_n$ and $M_n$ are given in \eqref{eq: MAB}. Therefore, we need to understand thoroughly the asymptotic behaviour of $A_n(x)M_n(x)-B_n^2(x)$ and of $M_n(x)$ in the interval $[n^{-3\gamma/4},1]$. This will be the content of Proposition \ref{prop: M}. Its proof requires a series of technical lemmas and will be presented in Appendix \ref{sec: Appendix}.

We now follow the strategy, starting with Jensen's inequality.
\begin{lemma}[Jensen's inequality]
\label{lem: Jensen}
Let $f$ be an entire complex function $f$ and $R,r>0$. The number of roots of $f$ in $\B(r) =\{z \in \C : |z| \leq r \}$, denoted by $N_f(r)$, satisfies 
\ben{ \label{nfr}
N_f(r) \leq \frac{\log \tfrac{M_R}{M_r}}{ \log \tfrac{R^2+r^2}{2Rr}},
}  
where $M_t=\max_{|z| \leq t} |f(z)|$ for $t>0$. 
\end{lemma}
An elementary proof of Jensen's inequality can be found in \cite[Section 15.5]{nguyen2017roots}. Now we show that $\mathbb{E} N_n^{(\gamma)}(0,n^{-3\gamma/4})$ is negligible as an interesting application of Jensen's inequality.  
\begin{proposition} 
\label{prop: fisrt interval}
We have 
$$\mathbb{E} N_n^{(\gamma)}(0,n^{-3\gamma/4})=o(\sqrt{n}).$$
\end{proposition}
\begin{proof}
We aim to apply  \eqref{nfr} to $P_n^{(\gamma)}(z)$, which is indeed an entire function. Let $r=n^{-3\gamma/4}$ and $R=n^{-2\gamma/3}$. Then 
\ben{
\log \frac{R^2+r^2}{2Rr}	\asymp \log n.
} 
Moreover, 
\bean{
M_r = \max_{|z| \leq r} |P_n^{(\gamma)}(z)| \geq |P(0)| = |\xi_0|, 
}
and 
\bean{ \label{MR}
&M_R = \max_{|z| \leq R} |P_n^{(\gamma)}(z)| & \displaystyle \leq \sum_{i=0}^n |\xi_i| R^i \binom{n}{i}^{\gamma} \leq \max_{0\leq i \leq n} |\xi_i| \times \displaystyle \sum_{i=0}^n \left( \sum_{i=0}^n   R^{i/\gamma} \binom{n}{i}\right)^{\gamma} \notag\\ 
&&\leq n \max_{0\leq i \leq n} |\xi_i| \times  (1+R^{1/\gamma})^{n\gamma}\notag\\
&& \leq n\max_{0\leq i \leq n} |\xi_i| \times \exp (\gamma \kO(n^{1/3})).
}
We define the event
\[\kE = \big \{ \max_{1\leq i \leq n} |\xi_i| \leq n \big \} \cap \{n^{-1} \leq |\xi_0| \leq n \}.  \]
Since $\{\xi_i\}_{i=0,\ldots,n}$ are standard normal i.i.d. random variables,
\ben{ \label{pne}
\P(\kE) \geq 1- \kO(1/n).
}
By combining \eqref{nfr}--\eqref{MR}, we obtain 
\ben{ \label{ene}
N_n^{(\gamma)}(r) \1(\kE) \leq  \frac{C n^{1/3}}{\log n},
}
for some positive constant $C$, where $N_n^{(\gamma)}(r)$ is the number of roots of $P_n^{(\gamma)}$ in the ball $\mathbb{B}(r)$ defined above. We notice also that $N_n^{(\gamma)}(r) \leq n$. Therefore, by \eqref{pne} and \eqref{ene}
\bea{
\E N_n^{(\gamma)}(0,n^{-3\gamma/4}) \leq \E N_n^{(\gamma)}(r) &=& \E[N_n^{(\gamma)}(r) \1(\kE)] +  \E[N_n^{(\gamma)}(r) \1(\kE^c)] \notag \\
& \leq & \frac{Cn^{1/3}}{\log n} + n \P(\kE^c) \leq C n^{1/3}.
}
As a consequence, we obtain
\ben{
\E N_n^{(\gamma)}(0,n^{-3\gamma/4}) = o(\sqrt{n}).
}
\end{proof}
As already mentioned, the following proposition characterizes precisely the asymptotic behaviour of $A_nM_n-B_n^2$ and of $M_n$, the two quantities appearing in \eqref{eq: temp1}. The proof of this proposition is presented in Appendix \ref{sec: Appendix}.
\begin{proposition}
\label{prop: M}
 If $1\geq x \geq \frac{(\log n)^{4\gamma}}{n^{\gamma}}$ then
\begin{equation}\label{termM} 
 \displaystyle M_n(x)=\sum_{i=0}^n \binom{n}{i}^{2\gamma}x^{2i}=\binom{n}{i_{\gamma,x}}^{2\gamma}x^{2i_{\gamma,x}} \times(\sqrt{ \pi} +o(1)) \sqrt{\frac{n x^{1/\gamma}}{\gamma (1+x^{1/\gamma})^2 }},
 \end{equation}
and 
\begin{equation}\label{termABM}
A_n(x)M_n(x) -B^2_n(x) =\binom{n}{i_{\gamma,x}}^{2\gamma}x^{4i_{\gamma,x}-2} \times \left(\frac{ \pi}{2} +o(1)\right) \left(\frac{n x^{1/\gamma}}{\gamma (1+x^{1/\gamma})^2 }\right)^2,
\end{equation}
where $i_{\gamma,x} =[nt_{\gamma,x}]$ with  $t_{\gamma,x}= \tfrac{x^{1/\gamma}}{1+x^{1/\gamma}}$.
\end{proposition}

We are now ready to prove the asymptotic behaviour of $\E N_n^{(\gamma)}$ (the first part of \eqref{eq: main asymptotic behaviour} in Theorem \ref{thm: asymptotic}).
\begin{proof}[Proof of asymptotic formula of $\E N_n^{(\gamma)}$]
From Proposition \ref{prop: fisrt interval}, Proposition \ref{prop: M} and Kac-Rice formula, we get
\begin{align*}
\mathbb{E} N_n^{(\gamma)}(0,1) & \displaystyle =\E N_n^{(\gamma)}(0,n^{-3\gamma/4}) +\E N_n^{(\gamma)}(n^{-3\gamma/4},1) = \frac{1}{\pi} \int_{n^{-3\gamma/4}}^1 \frac{\sqrt{A_n(x)M_n(x) -B^2_n(x)}}{M_n(x)} dx +o(\sqrt{n})\\
 & \displaystyle = \frac{\sqrt{n}}{\sqrt{2}\pi} \int_0^1 \frac{x^{\frac{1}{2\gamma}-1}}{\sqrt{\gamma} \left(1+x^{1/\gamma}\right)} dx +o(\sqrt{n}) =  \frac{\sqrt{n}}{\sqrt{2}\pi} \times \frac{2\sqrt{\gamma}\pi}{4}+o(\sqrt{n}) = \frac{\sqrt{2\gamma n}}{4}+o(\sqrt{n}),
\end{align*}
where the last line follows from the change of variable $u=x^{1/(2\gamma)}$ and the equality $\displaystyle \int_0^1 \frac{du}{1+u^2}=\frac{\pi}{4}$. Hence
$$
\mathbb{E} N_n^{(\gamma)}=4\mathbb{E} N_n^{(\gamma)}(0,1)=\sqrt{2\gamma n}+o(\sqrt{n}).
$$
The proof is complete.
\end{proof}

\label{sec:asymptotic of ENgamma}
\section{Asymptotic behaviour of $E N_n^{(\alpha,\beta)}$}
\label{sec: asymptotic of ENalpha}
This section deals with the asymptotic formula of $\E N_n^{(\alpha,\beta)}$ (the second part of \eqref{eq: main asymptotic behaviour} in Theorem \ref{thm: asymptotic}).  The strategy of the proof is as follows. We will first relate the asymptotic behaviour of $
\E N_n^{(\alpha,\beta)}$ for general $(\alpha,\beta)$ to that of $\E N_n^{(0,0)}$ for $\alpha=\beta=0$. We then exploit the relation that $\E N_n^{(0,0)}=\E N_n^{(1)}$ and use the result from the previous section. \\

\textit{The negligible interval $[0,n^{-3/4}]$}. We use the same argument as in Proposition \ref{prop: fisrt interval}. The estimate for $M_R$ can be replaced as 
  \bea{ 
  	M_R = \max_{|z| \leq R} |P_n^{(\alpha,\beta)}(z)| &  \leq &\sum_{i=0}^n |\xi_i| R^i \binom{n+\alpha}{n-i}^{\tfrac{1}{2}} \binom{n+\beta}{i}^{\tfrac{1}{2}}\\
  	 &\leq& \max_{0\leq i \leq n} |\xi_i| (n+|\alpha|)^{|\alpha|} (n+|\beta|)^{|\beta|} \times \displaystyle \sum_{i=0}^n     R^{i} \binom{n}{i} \notag\\ 
  	 	& \leq& \max_{0\leq i \leq n} |\xi_i| \times \exp (\kO(n^{1/3})),
  }
  where for the second line we used the inequality that $\binom{n+\alpha}{k} \leq \binom{n}{k} (n+|\alpha|)^{2|\alpha|}$. Then by repeating the same argument in Proposition \ref{prop: fisrt interval}, we can show that
   \ben{ \label{nabn}
   \E[N_n^{(\alpha,\beta)}(0,n^{-3/4})] = o(\sqrt{n}).
   }

\textit{The main interval $[n^{-3/4},1]$}. We first study the coefficients. It follows from  Stirling formula that as $i\wedge (n-i) \rightarrow \infty$,
\bean{
a_i^{(\alpha, \beta)} &=& \binom{n+\alpha}{n-i} \binom{n+\beta}{i} \notag\\
&=& (1+o(1)) \sqrt{\frac{(n+\alpha)(n+\beta)}{4 \pi^2 i(i+\alpha) (n-i) (n+\beta-i)} } \exp \left( (n+\alpha) I(\tfrac{\alpha +i}{n+\alpha}) + (n+\beta) I(\tfrac{i}{n+\beta}) \right) \notag \\
&=&(1+o(1)) \frac{n}{2 \pi i (n-i)  } \exp \left( (n+\alpha) I(\tfrac{\alpha +i}{n+\alpha}) + (n+\beta) I(\tfrac{i}{n+\beta}) \right),\notag
}	
where $I(t)=-t\log t + (t-1) \log (1-t)$. By Taylor expansion,
\bea{
I(\tfrac{i+\alpha}{n+\alpha}) &=& I(\tfrac{i}{n}) + I'(\tfrac{i}{n}) \left( \frac{i+\alpha}{n+\alpha} -\frac{i}{n} \right) + \kO(I''(i/n)) n^{-2}\\
&=& I(\tfrac{i}{n}) + I'(\tfrac{i}{n}) \frac{\alpha (n-i)}{n^2} +\kO(\tfrac{1}{i(n-i)}).\notag
}	
Note that $I''(t)=-t^{-1}(1-t)^{-1}$. Therefore, as $i\wedge (n-i) \rightarrow \infty$,
\bean{
(n+\alpha)I(\tfrac{i+\alpha}{n+\alpha}) - n I(\tfrac{i}{n}) = (1+o(1)) \left( \alpha I(\tfrac{i}{n}) + \alpha I'(\tfrac{i}{n}) \frac{ (n-i)}{n} \right). 
}
Similarly,
\bean{
	(n+\beta)I(\tfrac{i}{n+\beta}) - n I(\tfrac{i}{n}) = (1+o(1)) \left( \beta I(\tfrac{i}{n}) - \beta I'(\tfrac{i}{n}) \frac{ i}{n} \right).\notag
}
Hence,
\bean{
a_i^{(\alpha, \beta)} = (1+o(1)) \frac{n}{2 \pi i (n-i)  }  \exp \left( (\alpha+\beta)I(\tfrac{i}{n}) +I'(\tfrac{i}{n}) \frac{\alpha (n-i)-\beta i}{n}  \right) \exp(2 n I(\tfrac{i}{n})),\notag
}
so that 
\bea{
a_i^{(\alpha, \beta)}&=&(1+o(1)) \exp \left( (\alpha+\beta)I(\tfrac{i}{n}) +I'(\tfrac{i}{n}) \frac{\alpha (n-i)-\beta i}{n}  \right)  a_i^{(0, 0)}\\
&=& (1+o(1)) h_{(\alpha,\beta)}(\tfrac{i}{n}) a_i^{(0,0)},
}
where for $t \in (0,1)$
 \[ h_{(\alpha,\beta)}(t)=(\alpha+\beta) I(t) + I'(t) (\alpha (1-t) -\beta t). \]
Suppose that $x \in [n^{-3/4},1]$.  In the case $(\alpha, \beta) =(0,0)$, or equivalently the case $\gamma =1$, in Lemma \ref{lem: lem1} below, we show that the terms $a_i^{(0,0)} x^i$ attain the maximum around $i=i_x\pm i_x^{3/4}$ with $i_x = [nx/(x+1)]$, and the other terms are negligible. Here, the asymptotic behavior of $a_i^{(\alpha,\beta)}$ differs from that of $a_i^{(0,0)}$ only on the term $h_{(\alpha,\beta)}(\tfrac{i}{n})$ which is minor compared with $a_i^{(0,0)}$. Hence, using exactly the same analysis in Lemma \ref{lem: lem1}, we can also show that these terms $a_i^{(\alpha,\beta)} x^i$ when $|i-i_x| \geq i_x^{3/4}$ are negligible.  Therefore,
  \bea{
  M_n^{(\alpha,\beta)}(x) &=& (1+o(1)) \sum_{i:|i-i_x| \leq i_x^{3/4}}  a_i^{(\alpha,\beta)}x^i =  (1+o(1)) \sum_{i:|i-i_x| \leq i_x^{3/4}} h_{(\alpha,\beta)}(\tfrac{i}{n}) a_i^{(0,0)}x^i \\
  &=&(1+o(1)) h_{(\alpha,\beta)}( \tfrac{x}{x+1}) \sum_{i:|i-i_x| \leq i_x^{3/4}}  a_i^{(0,0)} x^i = (1+o(1)) h_{(\alpha,\beta)}( \tfrac{x}{x+1})M_n^{(0,0)}(x),
  }   
since when  $|i-i_x| \leq i_x^{3/4}$, we have  $h_{(\alpha,\beta)}(\tfrac{i}{n}) =(1+o(1))h_{\alpha,\beta}(\tfrac{x}{x+1})$. Similarly, 
 \bea{
 A_n^{(\alpha,\beta)} (x)= (1+o(1))h_{(\alpha,\beta)}( \tfrac{x}{x+1})A_n^{(0,0)}(x), \quad B_n^{(\alpha,\beta)}(x) = (1+o(1))h_{(\alpha,\beta)}( \tfrac{x}{x+1})B_n^{(0,0)}(x).
 } 
  Thus for $x \in [n^{-3/4},1]$,
 \be{
 	f_n^{(\alpha,\beta)}(x) = (1+o(1))f_n^{(0,0)}(x),
 }
 and hence
 \ben{ \label{nabm}
 	\E N_n^{(\alpha,\beta)}(n^{-3/4},1) =(1+o(1)) \E N_n^{(0,0)}(n^{-3/4},1) = (1+o(1)) \frac{\sqrt{2n}}{4}.
 }
Combining \eqref{nabn} and \eqref{nabm}, we obtain that $\E N_n^{(\alpha,\beta)}(0,1) =(1+o(1))\tfrac{\sqrt{2n}}{4}$, and hence
 \be{
 \E N_n^{(a,b)} =4 \E N_n^{(a,b)}(0,1) = (1+o(1)) \sqrt{2n}.
  }

%
%
%
%
\section{Summary and outlook}
\label{sec: summary}

In this paper, we have proved asymptotic formulas for the expected number of real roots of two general class of random polynomials. As a consequence, we have obtained an asymptotic formula for the expected number of internal equilibria in multi-player two-strategy random evolutionary games. Our results deepen the connection between evolutionary game theory and random polynomial theory which was discovered previously in \cite{DuongHanJMB2016,DuongTranHanJMB}. Below we discuss some important directions for future research.

\textit{Extensions to other models in EGT}. The class of random polynomials that we studied in this paper arises from the replicator dynamics. It would be interesting to generalize our results to more complex models in evolutionary game theory and population dynamics. The most natural model to study next is the replicator-mutator dynamics where mutation is present. Equilibria for the replicator-mutator dynamics are positive roots of a much more complicated class of random polynomials, which depend on the mutation. Studying the effect of mutation on the equilibrium properties, in particular on the expected number of internal equilibria, is a challenging problem, see \cite{DuongHanDGA2020} for an initial attempt. One can also ask whether our results can be extended to multi-player multi-strategy evolutionary games whose equilibria are positive roots of a system of random polynomials. In this case, the assumption that the gain sequence is independent is not realistic from evolutionary game theory's point of view (see \cite[Remark 4]{DH15} for a detailed explanation). Therefore, one needs to deal with a dependent system of random polynomials, which is very challenging.

\textit{Universality and other statistical properties}. The assumption that the random variables $\{\xi_i\}_{i=0}^n$ are Gaussian distributions is crucial in the present paper. It allowed us to employ the fundamental tool of Kac-Rice formula in Section \ref{sec: Kac-Rice}. What happens if $\{\xi_i\}$ are not Gaussian? Very recently, it has been shown that \textit{universality phenomena} hold for three classes of random polynomials: Kac polynomials, elliptic polynomials, and Weyl polynomials (recall the Introduction for their explicit expressions) \cite{TaoVu15, NNV2016, Do2018}. The universality results state that the expectation of the number of real roots of these classes of random polynomials depend only on the mean and variance of the coefficients $\{\xi_i\}$ but not on their type of the distributions. It would be very interesting to obtain such a universality theorem for the class of random polynomials arising from evolutionary game theory studied in this paper. The  distributions  of  the  roots  in different classes are different, and the methods to study them need to be tailored to each of the class. It remains elusive to us whether the techniques in \cite{TaoVu15, NNV2016, Do2018} can be applied to the class of random polynomials in this paper. Furthermore, studying other statistical properties such as central limit theorem and the distribution of the number of equilibria also demands future investigations, see for instance \cite{CanDuongPham2019} for a characterization of the probability that a multiplayer random evolutionary random game has no internal equilibria.
\section{Appendix}
\label{sec: Appendix}
In this appendix, we present detailed computations and proofs of technical results used in previous sections.
\subsection{Proof of Lemma \ref{lem: relation Mn vs Jn} and detailed computations of $f_n(0)$ and $f_n(1)$,}
In this section, we prove Lemma \ref{lem: relation Mn vs Jn} and compute $f_n(0)$ and $f_n(1)$.
\begin{proof}[Proof of Lemma \ref{lem: relation Mn vs Jn}]
It follows from the definition of Jacobi polynomial \eqref{eq: Jacobi} that for any $q\in\mathbb{R}$
\begin{align*}
J_n^{\alpha,\beta}\left(\frac{1+q}{1-q}\right)&=\frac{1}{2^n}\sum_{i=0}^n\begin{pmatrix}
n+\alpha\\
n-i
\end{pmatrix}\begin{pmatrix}
n+\beta\\
i
\end{pmatrix}\left(\frac{1+q}{1-q}-1\right)^i\left(\frac{1+q}{1-q}+1\right)^{n-i}
\\&=\frac{1}{2^n}\sum_{i=0}^n\begin{pmatrix}
n+\alpha\\
n-i
\end{pmatrix}\begin{pmatrix}
n+\beta\\
i
\end{pmatrix}\left(\frac{2q}{1-q}\right)^{i}\left(\frac{2}{1-q}\right)^{n-i}
\\&=\frac{1}{(1-q)^n}\sum_{i=0}^n\begin{pmatrix}
n+\alpha\\
n-i
\end{pmatrix}\begin{pmatrix}
n+\beta\\
i
\end{pmatrix} q^{n-i}.
\end{align*}
Taking $q=x^2$ yields the statement of Lemma \ref{lem: relation Mn vs Jn}.
\end{proof}
Next we compute $f_n(0)$ and $f_n(1)$. We have
\begin{equation}
M_n(0)=\begin{pmatrix}
n+\alpha\\n
\end{pmatrix},\quad A_n(0)=\begin{pmatrix}
n+\alpha\\n-1
\end{pmatrix}\begin{pmatrix}
n+\beta\\1
\end{pmatrix}=(n+\beta)\begin{pmatrix}
n+\alpha\\n-1
\end{pmatrix},\quad B_n(0)=0.
\end{equation}
Thus
$$
f_n(0)^2=\frac{A_n(0)M_n(0)-B_n(0)^2}{M_n(0)^2}=\frac{A_n(0)}{M_n(0)}=\frac{n(n+\beta)}{\alpha+1},
$$
that is
$$
f_n(0)=\sqrt{\frac{n(n+\beta)}{1+\alpha}}.
$$
Next we compute $f_n(1)$. We have 
\begin{equation*}
M_n(1)=\sum_{i=0}^n\begin{pmatrix}
n+\alpha\\n-i
\end{pmatrix}\begin{pmatrix}
n+\beta\\i
\end{pmatrix}=\begin{pmatrix}
2n+\alpha+\beta\\n
\end{pmatrix}
\end{equation*}
Using the following formula for the derivative of Jacobi polynomials \cite[Section 4.5]{szego1975book}
\begin{align*}
&(2n+\alpha+\beta)(1-x^2)\frac{d}{dx}J_n^{(\alpha,\beta)}(x)
\\&=-n\Big((2n+\alpha+\beta)x+\beta-\alpha\Big)J_n^{(\alpha,\beta)}(x)+2(n+\alpha)(n+\beta)J_{n-1}^{(\alpha,\beta)}(x)
\end{align*}
and Lemma \ref{lem: relation Mn vs Jn} we obtain the following formula for the derivative of $M_n^{\alpha,\beta}(x)$.
%
%
\begin{equation}
\label{eq: derivative}
x(2n+\alpha+\beta)M_n'(x)=\Big(n(2n+\alpha+\beta)+\beta-\alpha\Big)M^{(\alpha,\beta)}_n(x)-2(1-x^2)(n+\alpha)(n+\beta)M^{(\alpha,\beta)}_{n-1}(x).
\end{equation}
Applying \eqref{eq: derivative} for $x=1$, we obtain
\begin{align*}
B_n(1)=\frac{1}{2}M_n'(1)=\frac{n(2n+\alpha+\beta)+\beta-\alpha)M_n(1)}{2(2n+\alpha+\beta)}=\frac{1}{2}\Big(n+\frac{\beta-\alpha}{2n+\alpha+\beta}\Big)\begin{pmatrix}
2n+\alpha+\beta\\n
\end{pmatrix}.
\end{align*}
Taking derivative of \eqref{eq: derivative} we have
\begin{align*}
(2n+\alpha+\beta)(M_n'(x)+xM_n''(x))&=(n(2n+\alpha+\beta)+\beta-\alpha)M_n'(x)
\\&\qquad-2(n+\alpha)(n+\beta)\Big[-2xM_{n-1}(x)+(1-x^2)M_{n-1}'(x)\Big].
\end{align*}
It follows that
\begin{align*}
A_n(1)&=\frac{1}{4}(M_n'(1)+M_n''(1))
\\&=\frac{1}{4(2n+\alpha+\beta)}\Big[(n(2n+\alpha+\beta)+\beta-\alpha)M_n'(1)+4(n+\alpha)(n+\beta)M_{n-1}(1)\Big]
\\&=\frac{1}{4}\Big(n+\frac{\beta-\alpha}{2n+\alpha+\beta}\Big)^2\begin{pmatrix}
2n+\alpha+\beta\\
n
\end{pmatrix}+\frac{(n+\alpha)(n+\beta)}{2n+\alpha+\beta}\begin{pmatrix}
2n+\alpha+\beta-2\\ 
n-1
\end{pmatrix}.
\end{align*}
Hence
\begin{align*}
f_n(1)&=\frac{\sqrt{A_n(1)M_n(1)-B_n(1)^2}}{M_n(1)}
\\&=\sqrt{\frac{(n+\alpha)(n+\beta)}{2n+\alpha+\beta}}\sqrt{\frac{\begin{pmatrix}
2n+\alpha+\beta-2\\ 
n-1
\end{pmatrix}}{\begin{pmatrix}
2n+\alpha+\beta\\ 
n
\end{pmatrix}}}
\\&=\frac{1}{2n+\alpha+\beta}\sqrt{\frac{n(n+\alpha)(n+\beta)(n+\alpha+\beta)}{2n+\alpha+\beta-1}}.
\end{align*}
In particular, when $\alpha=\beta$,
$$
f_n(1)=\frac{1}{2}\sqrt{\frac{n(n+2\alpha)}{2n+2\alpha-1}}.
$$
\subsection{Proof of Proposition \ref{prop: M}}
In this section we prove Proposition \ref{prop: M}. The proof will be established after a series of technical lemmas.
We start with the following lemma that provides an estimate for a power of the binomial coefficient, which is a key factor appearing in the expressions of $A_n, B_n$ and $M_n$.
\begin{lemma}
\label{lem: lem0}
For $0<t<1$and $x>0$ we define $I(t):= t\log\frac{1}{t} +(1-t)\log\frac{1}{1-t} $ and $J_{\gamma ,x}(t):= \gamma I(t) + t\log x$. Then
\begin{equation} \label{stirl}
\begin{pmatrix}
n\\i
\end{pmatrix}^\gamma x^i = \left(\frac{n}{2\pi i (n-i)}\right)^{\gamma/2} \left( 1+O\left( \frac{1}{i}+\frac{1}{n-i}\right)\right)^{\gamma} e^{n J_{\gamma ,x}\left( \frac{i}{n}\right)}.
\end{equation}
\end{lemma}
\begin{proof}[Proof of Lemma \ref{lem: lem0}]
It follows from Stirling formula
\begin{eqnarray*}
i! = \sqrt{2\pi i}(1+\kO(i^{-1})) \left(\frac{i}{e}\right)^i
\end{eqnarray*}
that
$$
\begin{pmatrix}
n\\i
\end{pmatrix} = \sqrt{\frac{n}{2\pi i (n-i)}} \left( 1+O\left( \frac{1}{i}+\frac{1}{n-i}\right)\right) e^{n I\left( \frac{i}{n}\right)},
$$
where 
$$
I(t)=t\log\frac{1}{t}+(1-t)\log\frac{1}{1-t}
$$
Therefore,
$$
\begin{pmatrix}
n\\i
\end{pmatrix}^\gamma x^i = \left(\frac{n}{2\pi i (n-i)}\right)^{\gamma/2} \left( 1+O\left( \frac{1}{i}+\frac{1}{n-i}\right)\right)^{\gamma} e^{n J_{\gamma ,x}\left( \frac{i}{n}\right)},
$$
which is the statement of the lemma.
\end{proof}

\begin{lemma}
\label{lem: lem1}
We define 
\begin{equation} 
\label{chisoi}
t_{\gamma ,x}:=\frac{x^{1/\gamma}}{1+x^{1/\gamma}}\quad \text{and}\quad i_{\gamma ,x}:= \lfloor n t_{\gamma ,x}\rfloor,
\end{equation}
We note that $t_{\gamma ,x}$ is the unique solution of the equation $J'_{\gamma ,x}(t)=0$, where
\begin{equation} \label{hamj}
J'_{\gamma ,x}(t)= \gamma \log \frac{1-t}{t}+\log x \; \; \mbox{and} \; \; J''_{\gamma ,x}(t)= \frac{-\gamma}{t(1-t)}.
\end{equation}
Assume that $1\geq x \geq (\log n)^{4\gamma}/n^\gamma$. 
\begin{itemize}
\item[a.] If $|i-i_{\gamma ,x}|\geq i_{\gamma ,x}^{3/4}$ then
$$\displaystyle \frac{\binom{n}{i}^\gamma x^i}{\binom{n}{i_{\gamma ,x}}^\gamma x^{i_{\gamma ,x}}} \leq \frac{1}{n^{10}}.$$

\item[b.] If $|i-i_{\gamma ,x}|< i_{\gamma ,x}^{3/4}$ then
$$\displaystyle \frac{\binom{n}{i}^\gamma x^i}{\binom{n}{i_{\gamma ,x}}^\gamma x^{i_{\gamma ,x}}} = \left(1+\kO\left(\frac{1}{\log n}\right)\right) \exp \left\{ \left[ J_{\gamma ,x}''\left(t_{\gamma ,x}\right) + \kO\left(\frac{(n t_{\gamma ,x})^{3/4}}{nt_{\gamma ,x}^2}\right) \right]\frac{(i-i_{\gamma ,x})^2}{2n} \right\}.$$
\end{itemize}
\end{lemma}
\begin{proof}[Proof of Lemma \ref{lem: lem1}]

 a.   Since $\displaystyle \binom{n}{i} \leq \exp(nI(i/n))$, using (\ref{stirl}) we have

\begin{align*}
\displaystyle \frac{\binom{n}{i}^\gamma x^i}{\binom{n}{i_{\gamma ,x}}^\gamma x^{i_{\gamma ,x}}} & \leq \displaystyle (2\pi i_{\gamma ,x})^{\gamma/2} \exp \left( n \left[ J_{\gamma ,x}(i/n) -J_{\gamma ,x}(i_{\gamma ,x}/n) \right]\right)\\
& \leq \displaystyle (2\pi i_{\gamma ,x})^{\gamma/2} \exp \left( n \left[ J_{\gamma ,x}\left( \frac{i_{\gamma ,x} \pm i_{\gamma ,x}^{3/4}}{n}\right) -J_{\gamma ,x}(i_{\gamma ,x}/n) \right]\right),
\end{align*}
where the second lines follows from the fact that the function $J_{\gamma ,x}(t)$ is concave and attains maximum at $t_{\gamma ,x}$.
 By Taylor expansion, there exists $\theta \in \left(\frac{i_{\gamma ,x} - i_{\gamma ,x}^{3/4}}{n} , \frac{i_{\gamma ,x} + i_{\gamma ,x}^{3/4}}{n}  \right)$ such that
 $$J_{\gamma ,x}\left( \frac{i_{\gamma ,x} \pm i_{\gamma ,x}^{3/4}}{n}\right) -J_{\gamma ,x}\left( \frac{i_{\gamma ,x}}{n}\right)=\frac{ \pm i_{\gamma ,x}^{3/4}}{n} J'_{\gamma ,x}\left( \frac{i_{\gamma ,x}}{n}\right)+ J''_{\gamma ,x}(\theta)\frac{i_{\gamma ,x}^{3/2}}{2n^2}.$$
 Notice that
\begin{eqnarray}
\Big |J'_{\gamma ,x}\left(\frac{i_{\gamma ,x}}{n}\right) \Big |= \Big |J'_{\gamma ,x}\left(\frac{i_{\gamma ,x}}{n}\right)-J'_{\gamma ,x}\left(t_{\gamma ,x}\right) \Big | &\leq& \sup \limits_{y \in (\tfrac{i_{\gamma ,x}}{n}, t_{\gamma ,x})} |J''_{\gamma ,x}(y)| \Big |\frac{i_{\gamma ,x}}{n}-t_{\gamma ,x} \Big | \notag \\
&\leq& \frac{1}{n}\sup \limits_{y \in (\tfrac{i_{\gamma ,x}}{n}, t_{\gamma ,x})} \Big | \frac{-\gamma}{y(1-y)}\Big | \leq \frac{C}{nx^{1/\gamma}}. \label{jpx}
\end{eqnarray}
Combining this with the fact that
\begin{eqnarray*}
J_{\gamma ,x}''(\theta)=\frac{-\gamma}{\theta(1-\theta)} \leq \frac{-\gamma}{\theta} \leq \frac{-n\gamma}{i_{\gamma ,x}+i_{\gamma ,x}^{3/4}} \leq \frac{-cn}{i_{\gamma ,x}},
\end{eqnarray*}
we get that for $n$ large enough,
\begin{eqnarray*}
\displaystyle \frac{\binom{n}{i}^\gamma x^i}{\binom{n}{i_{\gamma ,x}}^\gamma x^{i_{\gamma ,x}}} &\leq   (2\pi i_{\gamma ,x})^{\gamma/2}  \exp \left( \frac{C i_{\gamma ,x}^{3/4}}{nx^{1/\gamma}} - ci_{\gamma ,x}^{1/2}\right)\\
& \displaystyle  \leq (2\pi n)^{\gamma/2} \exp \left(  - c' (\log n)^2\right) \leq \frac{1}{n^{10}},
\end{eqnarray*}
where in the last line we use the estimate $i_{\gamma ,x} \simeq n x^{1/\gamma} =(\log n)^4$. This ends the proof of Part a.

b. Suppose that $|i-i_{\gamma ,x}|< i_{\gamma ,x}^{3/4}$. By using \eqref{stirl} and Taylor expansion,
\begin{eqnarray*}
&& \frac{\binom{n}{i}^{\gamma} x^{i}}{\binom{n}{i_{\gamma ,x}}^{\gamma} x^{2_{\gamma ,x}}} =  \left[(1+\kO(i_{\gamma ,x}^{-1})) \frac{i_{\gamma ,x}(n-i_{\gamma ,x})}{i(n-i)} \right]^{\gamma/2} \exp \left( n \left[ J_x\left(\frac{i}{n}\right) -J_x\left(\frac{i_x}{n}\right) \right] \right)\\
&=& \left(1+\kO\left(\frac{1}{\log n}\right)\right)  \exp \left( n \left[ J_{\gamma ,x}'\left(\frac{i_{\gamma ,x}}{n}\right) \frac{(i-i_{\gamma ,x})}{n}+  J_{\gamma ,x}''\left(\frac{i_{\gamma ,x}}{n}\right) \frac{(i-i_{\gamma ,x})^2}{2n^2} + J_{\gamma ,x}'''\left(\theta\right) \frac{(i-i_{\gamma ,x})^3}{6n^3}  \right]  \right),
\end{eqnarray*}
for some $\theta \in (\tfrac{i_{\gamma ,x}}{n}, \tfrac{i}{n})$. 

 Observe that 
\begin{itemize}
\item $\displaystyle \Big |J'_{\gamma ,x}\left(\frac{i_{\gamma ,x}}{n}\right) \Big | \leq \frac{C'}{nx^{1/\gamma}}\leq \frac{C}{i_{\gamma ,x}} $ as in the proof of Part a.
\item $J_{\gamma ,x}'''(y) = \kO(y^{-2})$ for all $y\in \mathbb{R}$. Therefore $J_{\gamma ,x}'''(\theta) = \kO(t_{\gamma ,x}^{-2})$.
\item $\displaystyle J_{\gamma ,x}''\left(\frac{i_{\gamma ,x}}{n}\right)=J_x''\left(t_{\gamma ,x}\right) + \kO\left(\frac{1}{t_{\gamma ,x}^2}\right) \left(\frac{i_{\gamma ,x}}{n}-t_{\gamma ,x}\right) =J_x''\left(t_{\gamma ,x}\right) + \kO\left(\frac{1}{nt_{\gamma ,x}^2}\right).$
\end{itemize}

Combining these estimates,
\begin{eqnarray*}
 n \left[ J_x\left(\frac{i}{n}\right) -J_x\left(\frac{i_x}{n}\right) \right]
& =& \displaystyle \kO\left(\frac{i_{\gamma ,x}^{3/4}}{i_{\gamma ,x}}\right)+    \left[ J_{\gamma ,x}''\left(t_{\gamma ,x}\right) + \kO\left(\frac{1}{nt_{\gamma ,x}^2}\right) + \kO\left(\frac{1}{t_{\gamma ,x}^2}\right)\frac{i-i_{\gamma ,x}}{n} \right]\frac{(i-i_{\gamma ,x})^2}{2n} \\
& = &\displaystyle \kO\left(i_{\gamma ,x}^{-1/4}\right)+    \left[ J_{\gamma ,x}''\left(t_{\gamma ,x}\right) + \kO\left(\frac{(n t_{\gamma ,x})^{3/4}}{nt_{\gamma ,x}^2}\right) \right]\frac{(i-i_{\gamma ,x})^2}{2n}.
\end{eqnarray*}
Then Part b. follows.
\end{proof}
\begin{lemma}
\label{lem: lem2}
 We have
\begin{equation}
A_n(x) M_n(x) - B_n(x)^2=\frac{1}{2} \sum_{i,j=0}^n  (i-j)^2  \binom{n}{i}^{2\gamma} \binom{n}{j}^{2\gamma} x^{2(i+j-1)}.
\end{equation}
\end{lemma}
\begin{proof}[Proof of Lemma \ref{lem: lem2}]
Using formulas of $M_n, A_n$ and $B_n$ given in \eqref{eq: MAB}, we obtain
\begin{align*}
A_n(x) M_n(x) - B_n(x)^2 &=\sum_{i,j=0}^n  i^2  \binom{n}{i}^{2\gamma} x^{2(i-1)} \binom{n}{j}^{2\gamma} x^{2 j} - \sum_{i,j=0}^n i  \binom{n}{i}^2 x^{2i-1} j  \binom{n}{j}^2 x^{2j-1} \notag \\& = \frac{1}{2} \sum_{i,j=0}^n  [i^2 + j^2 -2 ij]  \binom{n}{i}^{2\gamma} \binom{n}{j}^{2\gamma} x^{2(i+j-1)} \notag \\
&= \frac{1}{2} \sum_{i,j=0}^n  (i-j)^2  \binom{n}{i}^{2\gamma} \binom{n}{j}^{2\gamma} x^{2(i+j-1)}.
\end{align*}
The last step is the desired equality.
\end{proof}
The following lemma will be used to obtain asymptotic behaviour of $A_nM_n-B_n^2$ later on.
\begin{lemma}
\label{lem: lem3}
 Let $f(x,y)$ be a bivariate function such that $f(x,y)=\kO(x^2+y^2)$. Consider two sequences $(a_n)$ and $(b_n)$ such that $a_n \rightarrow 0$ and $\frac{b_n(\log n)^4}{a_n} \rightarrow 0$. Then for $k< i_{\gamma,x}-a_n^{-1}$ and $l> i_{\gamma,x}+a_n^{-1}$,
\begin{align*}
& \displaystyle  \sum_{i,j=k}^l f(i,j) \exp \left( -(a_n+b_n) \left( (i-i_{\gamma,x})^2+(j-i_{\gamma,x})^2\right)\right)\\
= & \displaystyle \left(1+\kO\left(\frac{1}{\log n}\right)\right) \sum_{{\substack{|i-i_{\gamma,x}|  < (a_nb_n)^{-1/4}\\ |j-i_{\gamma,x}| < (a_nb_n)^{-1/4}}}} f(i,j) \exp \left( -a_n \left( (i-i_{\gamma,x})^2+(j-i_{\gamma,x})^2\right)\right) +\kO(1).
\end{align*}
\end{lemma}
\begin{proof}[Proof of Lemma \ref{lem: lem3}]
Denote by $\theta = (a_nb_n)^{-1/4}$. If $|i-i_{\gamma,x}|\geq \theta$ or $|j-i_{\gamma,x}|\geq \theta$ then
$$ -(a_n+b_n) \left( (i-i_{\gamma,x})^2+(j-i_{\gamma,x})^2\right) \leq -a_n \theta^2 \leq -c(\log n)^2.$$
Therefore
$$\sum_{{\substack{|i-i_{\gamma,x}|  < (a_nb_n)^{-1/4} \; \textrm{or}\\ |j-i_{\gamma,x}| < (a_nb_n)^{-1/4}}}} f(i,j) \exp \left( -(a_n+b_n) \left( (i-i_{\gamma,x})^2+(j-i_{\gamma,x})^2\right)\right) \leq Cn^4 e^{-c(\log n)^2} =\kO(1).$$
Consider $|i-i_{\gamma,x}|< \theta$ and $|j-i_{\gamma,x}|< \theta$. Then 
$$b_n \left( (i-i_{\gamma,x})^2+(j-i_{\gamma,x})^2\right) = \kO (b_n \theta^2) =\kO \left( \frac{1}{(\log n)^2}\right).$$
It implies that 
$$\exp \left( -(a_n+b_n) \left( (i-i_{\gamma,x})^2+(j-i_{\gamma,x})^2\right)\right)= \left(1+\kO\left(\frac{1}{\log n}\right)\right) \exp \left( -a_n \left( (i-i_{\gamma,x})^2+(j-i_{\gamma,x})^2\right)\right).  $$
Then the result follows.
\end{proof}
\begin{lemma}
\label{lem: lem4}
 Let $g:\; \R \rightarrow \R$ be a differentiable function such that
$$\int_{\R}(|g(x)+|g'(x)|)(x^2+|x|+1)dx < \infty.$$
Then for any $K,l,m$ such that $K,\frac{l}{\sqrt{K}},\frac{m}{\sqrt{K}} \rightarrow \infty$, we have
\begin{equation}
\label{firstinter}
\sum_{i=-l}^m g\left( \frac{i}{\sqrt{K}}\right)= (1+o(1)) \sqrt{K} \int_{\R} g(x)dx,
\end{equation}
and
\begin{equation}
\label{secondinter}
\sum_{i,j=-l}^m (i-j)^2 g\left( \frac{i}{\sqrt{K}}\right)g\left( \frac{j}{\sqrt{K}}\right)= (1+o(1)) (\sqrt{K})^4 \int_{\R^2} (x-y)^2 g(x)g(y)dxdy.
\end{equation}
\end{lemma}
\begin{proof}[Proof of Lemma \ref{lem: lem4}]
We can rewrite (\ref{firstinter}) as
$$\underset{k\rightarrow \infty}{\lim} \frac{1}{\sqrt{K}}\sum_{i=-l}^m g\left( \frac{i}{\sqrt{K}}\right)=  \int_{\R} g(x)dx < \infty.$$
Indeed, for any $\epsilon >0$ there exists $N_{\epsilon} > \epsilon^{-1}$ such that
$$\int_{\R\setminus [-N_{\epsilon},N_{\epsilon}]} |g(x)|dx\leq \epsilon. $$
It is clear that
$$\underset{k\rightarrow \infty}{\lim} \frac{1}{\sqrt{K}}\sum_{-\sqrt{K}N_{\epsilon} \leq i \leq \sqrt{K}N_{\epsilon} } g\left( \frac{i}{\sqrt{K}}\right)=  \int_{-N_{\epsilon}}^{N_{\epsilon}} g(x)dx < \infty,$$
then there exists $K_{\epsilon}>0$ such that for any $K>K_{\epsilon}$,
$$\left|  \frac{1}{\sqrt{K}}\sum_{-\sqrt{K}N_{\epsilon} \leq i \leq \sqrt{K}N_{\epsilon} } g\left( \frac{i}{\sqrt{K}}\right)-  \int_{-N_{\epsilon}}^{N_{\epsilon}} g(x)dx\right| \leq \epsilon.$$
For the remainder term, using the fact that there exists $M>0$ such that $|g(x)|x^4 \leq M,\; \forall x\in \R$, 
\begin{align*}
&\displaystyle\left| \frac{1}{\sqrt{K}}\sum_{i\in [-l,m]\setminus [-\sqrt{K}N_{\epsilon},\sqrt{K}N_{\epsilon}] } g\left( \frac{i}{\sqrt{K}}\right)  \right|  \leq  \frac{2}{\sqrt{K}}\sum_{i>\sqrt{K}N_{\epsilon}} \frac{M}{(i/\sqrt{K})^4} \\
\leq &\displaystyle   2MK^{3/2} \int_{\sqrt{K}N_{\epsilon}}^{\infty} x^{-4}dx =\frac{2MK^{3/2}}{3 (\sqrt{K}N_{\epsilon})^3} = \frac{2M}{3N_{\epsilon}^3}\leq \epsilon.
\end{align*}
In conclusion, for any $\epsilon >0$ there exists $K_{\epsilon}>0$ such that for any $K>K_{\epsilon}$,
$$\left|  \frac{1}{\sqrt{K}}\sum_{i=-l}^mg\left( \frac{i}{\sqrt{K}}\right)-  \int_{\R} g(x)dx\right| \leq 3\epsilon.$$
It implies (\ref{firstinter}). By the same argument, we can prove (\ref{secondinter}).
\end{proof}
Finally, we now bring all previous technical lemmas to prove Proposition \ref{prop: M}.
\begin{proof}[Proof of Proposition \ref{prop: M}]
We have
$$M_n(x)= \sum_{|i-i_{\gamma,x}| \geq i_{\gamma,x}^{3/4}} \binom{n}{i}^{2\gamma}x^{2i}+\sum_{|i-i_{\gamma,x}| < i_{\gamma,x}^{3/4}}\binom{n}{i}^{2\gamma}x^{2i}=:M_{1,n}+M_{2,n}.$$
By Lemma \ref{lem: lem1} Part a.,
$$\frac{M_{1,n}}{\binom{n}{i_{\gamma,x}}^{2\gamma}x^{2i_{\gamma,x}}} \leq \frac{n}{n^{20}} \leq \frac{1}{n^{10}};$$
and by Lemma \ref{lem: lem1} Part b., Lemma \ref{lem: lem3} and Lemma \ref{lem: lem4},
\bea{
\frac{M_{2,n}}{\binom{n}{i_{\gamma,x}}^{2\gamma}x^{2i_{\gamma,x}}} &=& (1+o(1)) \sum_{|i-i_{\gamma, x}| < i_{\gamma, x}^{3/4}  } \exp \left( \left[ J_{\gamma ,x}''\left(t_{\gamma ,x}\right) + \kO\left(\frac{(n t_{\gamma ,x})^{3/4}}{nt_{\gamma ,x}^2}\right) \right]\frac{(i-i_{\gamma ,x})^2}{n} \right) \\
&=& (1+o(1)) \int_{\R}e^{-x^2}dx \times \sqrt{n/|J"(t_{\gamma,x})|} \\
&=& (\sqrt{ \pi} +o(1)) \sqrt{\frac{n x^{1/\gamma}}{\gamma (1+x^{1/\gamma})^2 }}.
}
Now according to Lemma \ref{lem: lem2}, we have
$$ A_n(x)M_n(x) -B^2_n(x)=\frac{1}{2}\sum_{i,j=0}^n (i-j)^2 \binom{n}{i}^{2\gamma}x^{2i}\binom{n}{j}^{2\gamma}x^{2i+2j-2}.$$
Then by the same argument as above, we get
\bea{
 A_n(x)M_n(x) -B^2_n(x)&=&\frac{1}{2}(1+o(1))\int_{\R^2}(x-y)^2 e^{-(x^2+y^2)}dxdy \times (n/|J"(t_{\gamma,x})|)^2 \\
 &=&  \left(\frac{ \pi}{2} +o(1)\right) \left(\frac{n x^{1/\gamma}}{\gamma (1+x^{1/\gamma})^2 }\right)^2,
}
which completes the proof.
\end{proof}

\bibliographystyle{plain}
\bibliography{GTbib}

\begin{thebibliography}{10}

\bibitem{Bach2006}
L.A. Bach, T.~Helvik, and F.B. Christiansen.
\newblock The evolution of n-player cooperation—threshold games and ess
  bifurcations.
\newblock {\em Journal of Theoretical Biology}, 238(2):426 -- 434, 2006.

\bibitem{BS86}
A.~T. Bharucha-Reid and M.~Sambandham.
\newblock {\em Random polynomials}.
\newblock Probability and Mathematical Statistics. Academic Press, Inc.,
  Orlando, FL, 1986.

\bibitem{Bogomolny1992}
E.~Bogomolny, O.~Bohigas, and P.~Leboeuf.
\newblock Distribution of roots of random polynomials.
\newblock {\em Phys. Rev. Lett.}, 68:2726--2729, May 1992.

\bibitem{broom:1997aa}
M.~Broom, C.~Cannings, and G.T. Vickers.
\newblock Multi-player matrix games.
\newblock {\em Bull. Math. Biol.}, 59(5):931--952, 1997.

\bibitem{CanDuongPham2019}
V.~H. Can, M.~H. Duong, and V.~H. Pham.
\newblock Persistence probability of a random polynomial arising from
  evolutionary game theory.
\newblock {\em Journal of Applied Probability}, 56(3):870–890, 2019.

\bibitem{Do2018}
Y.~Do, O.~Nguyen, and V.~Vu.
\newblock Roots of random polynomials with coefficients of polynomial growth.
\newblock {\em Ann. Probab.}, 46(5):2407--2494, 09 2018.

\bibitem{DH15}
M.~H. Duong and T.~A. Han.
\newblock On the expected number of equilibria in a multi-player multi-strategy
  evolutionary game.
\newblock {\em Dynamic Games and Applications}, pages 1--23, 2015.

\bibitem{DuongHanJMB2016}
M.~H. Duong and T.~A. Han.
\newblock Analysis of the expected density of internal equilibria in random
  evolutionary multi-player multi-strategy games.
\newblock {\em Journal of Mathematical Biology}, 73(6):1727--1760, 2016.

\bibitem{DuongHanDGA2020}
M.~H. Duong and T.~A. Han.
\newblock On equilibrium properties of the replicator–mutator equation in
  deterministic and random games.
\newblock {\em Dynamics Games and Applications}, (10):641–663, 2020.

\bibitem{DuongTranHanJMB}
M.~H. Duong, H.~M. Tran, and T.~A. Han.
\newblock On the distribution of the number of internal equilibria in random
  evolutionary games.
\newblock {\em Journal of Mathematical Biology}, Aug 2018.

\bibitem{DuongTranHanDGA}
M.~H. Duong, H.~M. Tran, and T.~A. Han.
\newblock On the expected number of internal equilibria in random evolutionary
  games with correlated payoff matrix.
\newblock {\em Dynamic Games and Applications}, Jul 2018.

\bibitem{EK95}
A.~Edelman and E.~Kostlan.
\newblock How many zeros of a random polynomial are real?
\newblock {\em Bull. Amer. Math. Soc. (N.S.)}, 32(1):1--37, 1995.

\bibitem{Emiris:2010}
I.~Z. Emiris, A.~Galligo, and E.~P. Tsigaridas.
\newblock Random polynomials and expected complexity of bisection methods for
  real solving.
\newblock In {\em Proceedings of the 2010 International Symposium on Symbolic
  and Algebraic Computation}, ISSAC '10, pages 235--242, New York, NY, USA,
  2010. ACM.

\bibitem{farahmand1998}
K.~Farahmand.
\newblock {\em Topics in Random Polynomials}.
\newblock Chapman \& Hall/CRC Research Notes in Mathematics Series. Taylor \&
  Francis, 1998.

\bibitem{fudenberg1992evolutionary}
D.~Fudenberg and C.~Harris.
\newblock Evolutionary dynamics with aggregate shocks.
\newblock {\em Journal of Economic Theory}, 57(2):420--441, 1992.

\bibitem{FyoKho2016}
Y.~V. Fyodorov and B.~A. Khoruzhenko.
\newblock Nonlinear analogue of the may-wigner instability transition.
\newblock {\em Proceedings of the National Academy of Sciences},
  113(25):6827--6832, 2016.

\bibitem{Galla2013}
T.~Galla and J.~D. Farmer.
\newblock Complex dynamics in learning complicated games.
\newblock {\em Proceedings of the National Academy of Sciences},
  110(4):1232--1236, 2013.

\bibitem{GT10}
C.~S. Gokhale and A.~Traulsen.
\newblock Evolutionary games in the multiverse.
\newblock {\em Proceedings of the National Academy of Sciences}, 2010.

\bibitem{gross2009generalized}
T.~Gross, L.~Rudolf, S.~A Levin, and U.~Dieckmann.
\newblock Generalized models reveal stabilizing factors in food webs.
\newblock {\em Science}, 325(5941):747--750, 2009.

\bibitem{Hauert2006}
C.~Hauert, F.~Michor, M.~A. Nowak, and M.~Doebeli.
\newblock Synergy and discounting of cooperation in social dilemmas.
\newblock {\em Journal of Theoretical Biology}, 239(2):195 -- 202, 2006.

\bibitem{May1972}
R.~May.
\newblock Will a large complex system be stable?
\newblock {\em Nature}, 238:413--414, 1972.

\bibitem{may2001stability}
R.~M. May.
\newblock {\em Stability and complexity in model ecosystems}, volume~6.
\newblock Princeton university press, 2001.

\bibitem{Newman2003}
M.~E.~J. Newman.
\newblock The structure and function of complex networks.
\newblock {\em SIAM Review}, 45(2):167--256, 2003.

\bibitem{NNV2016}
Hoi Nguyen, Oanh Nguyen, and Van Vu.
\newblock On the number of real roots of random polynomials.
\newblock {\em Communications in Contemporary Mathematics}, 18(04):1550052,
  2016.

\bibitem{nguyen2017roots}
Oanh Nguyen and Van Vu.
\newblock Roots of random functions: A general condition for local
  universality, 2017.

\bibitem{Pacheco2009}
J.~M. Pacheco, F.~C. Santos, M.~O. Souza, and B.~Skyrms.
\newblock Evolutionary dynamics of collective action in $n$-person stag hunt
  dilemmas.
\newblock {\em Proceedings of the Royal Society B: Biological Sciences},
  276(1655):315--321, 2009.

\bibitem{PENA2018}
Jorge Peña and Georg Nöldeke.
\newblock Group size effects in social evolution.
\newblock {\em Journal of Theoretical Biology}, 457:211 -- 220, 2018.

\bibitem{Pena2014}
Jorge Pe{\~n}a, Laurent Lehmann, and Georg N{\"o}ldeke.
\newblock Gains from switching and evolutionary stability in multi-player
  matrix games.
\newblock {\em Journal of Theoretical Biology}, 346:23 -- 33, 2014.

\bibitem{Pereda2019}
M.~Pereda, V.~Capraro, and A.~S\'{a}nchez.
\newblock Group size effects and critical mass in public goods games.
\newblock {\em Sci Rep}, (9):5503, 2019.

\bibitem{Pimm1984}
S.~Pimm.
\newblock The complexity and stability of ecosystems.
\newblock {\em Nature}, 307:321–326, 1984.

\bibitem{Shub1993}
M.~Shub and S.~Smale.
\newblock {\em Complexity of Bezout's Theorem II Volumes and Probabilities},
  pages 267--285.
\newblock Birkh{\"a}user Boston, Boston, MA, 1993.

\bibitem{szego1975book}
G.~Szeg{\H{o}}.
\newblock {\em Orthogonal Polynomials}.
\newblock American Math. Soc: Colloquium publ. American Mathematical Society,
  1975.

\bibitem{TaoVu15}
T.~Tao and V.~Vu.
\newblock Local universality of zeroes of random polynomials.
\newblock {\em International Mathematics Research Notices}, 2015(13):5053,
  2015.

\end{thebibliography}

\end{document}